\documentclass[11pt]{article}

\usepackage{amssymb, amsmath, amsthm, graphicx}
\usepackage[left=1in,top=1in,right=1in]{geometry}

\date{}

\theoremstyle{plain}
      \newtheorem{theorem}{Theorem}[section]
      \newtheorem{lemma}[theorem]{Lemma}

      \newtheorem{corollary}[theorem]{Corollary}

\theoremstyle{definition}
      \newtheorem{definition}[theorem]{Definition}

\theoremstyle{remark}

\def\cn{\mbox{\rm cr}}

\title{On the number of edges of separated multigraphs}
\author{Jacob Fox\thanks{Stanford University, Stanford, CA. Supported by a Packard Fellowship and by NSF award DMS-1855635. Email: {\tt jacobfox@stanford.edu.}} \and J\'anos Pach\thanks{R\'enyi Institute, Budapest and MIPT, Moscow. Supported by NKFIH grants K-176529, KKP-133864, Austrian Science Fund Z 342-N31, Ministry of Education and Science of the Russian Federation MegaGrant No. 075-15-2019-1926, ERC Advanced Grant ``GeoScape.'' Email:
{\tt pach@cims.nyu.edu}.}\and  Andrew Suk\thanks{Department of Mathematics, University of California at San Diego, La Jolla, CA, 92093 USA. Supported an NSF CAREER award, NSF award DMS-1952786, and an Alfred Sloan Fellowship. Email: {\tt asuk@ucsd.edu}.} }

\begin{document}

\maketitle

\begin{abstract}
We prove that the number of edges of a multigraph $G$ with $n$ vertices is at most $O(n^2\log n)$, provided that any two edges cross at most once, parallel edges are noncrossing, and the lens enclosed by every pair of parallel edges in $G$ contains at least one vertex. As a consequence, we prove the following extension of the Crossing Lemma of Ajtai, Chv\'atal, Newborn, Szemer\'edi and Leighton, if $G$ has $e \geq 4n$ edges, in any drawing of $G$ with the above property, the number of crossings is $\Omega\left(\frac{e^3}{n^2\log(e/n)}\right)$. This answers a question of Kaufmann {\em et al.}~and is tight up to the logarithmic factor.
\end{abstract}

\section{Introduction}

A \emph{topological graph} is a graph drawn in the plane such that its vertices are represented by points, and the edges are represented by simple continuous arcs connecting the corresponding
pairs of points. In notation and terminology, we do not distinguish between the vertices and the points representing them and the edges and the arcs representing them. The edges are allowed to intersect, but they cannot pass through any vertex other than their endpoints. If two edges share an interior point, then they must properly \emph{cross} at that point, {\em i.e.}, one edge passes from one side of the other edge to its other side.

A {\em multigraph} is a graph in which two vertices can be joined by several edges. Two edges that join the same pair of vertices are called {\em parallel}.
\smallskip

According to the {\em crossing lemma} of Ajtai, Chv\'atal, Newborn, Szemer\'edi~\cite{ACNS82} and Leighton~\cite{L83}, every topological graph $G$ with $n$ vertices and $e>4n$ edges has at least $c\frac{e^3}{n^2}$ edge crossings, where $c>0$ is an absolute constant. In notation, we have

\begin{equation}\label{eq1}
\cn(G)\ge c\cdot \frac{e^3}{n^2}.
\end{equation}

In a seminal paper which was an important step towards the solution of Erd\H os's famous problem on distinct distances \cite{GuK15}, Sz\'ekely \cite{Sz97} generalized the crossing lemma to multigraphs: for every topological multigraph $G$ with $n$ vertices and $e>4n$ edges, in which the {\em multiplicity} of every edge is at most $m$, we have
\begin{equation}\label{eq2}
\cn(G)\ge c\frac{e^3}{mn^2}.
\end{equation}
As the maximum multiplicity $m$ increases, (\ref{eq2}) gets weaker. However, as was shown in \cite{PT} and~\cite{KPTU}, under certain special conditions on the multigraphs, the inequality (\ref{eq1}) remains true, independently of $m$. Some related results were established in~\cite{PTT}. In all of these papers, one of the key elements of the argument was to find an analogue of Euler's theorem for the corresponding classes of ``nearly planar'' multigraphs.
\smallskip

Throughout this paper, we consider only {\em single-crossing} topological multigraphs, {\em i.e.}, we assume that any two edges cross {\em at most once}.  Hence, two edges that share an endpoint may also have a common interior point.  Two edges are said to be \emph{independent} if they do not share an endpoint, and they are called \emph{disjoint} if they are independent and do not cross.

\begin{definition}
A multigraph $G$ is called {\em separated} if no two parallel edges of $G$ cross, and the ``lens" enclosed by them has at least one vertex in its interior.
\end{definition}

It was conjectured in~\cite{KPTU} that any separated single-crossing topological multigraph with $n$ vertices has at most $O(n^2)$ edges. The aim of this note is to verify this conjecture apart from a logarithmic factor.

\begin{theorem}\label{main}
The number of edges of a separated single-crossing topological multigraph $G$ on $n$ vertices satisfies  $|E(G)| \leq 64n^{2}\log n.$
\end{theorem}

Note that in a separated multigraph, any pair of vertices can be connected by at most $n-1$ edges. This immediately implies the bound $|E(G)|\le {n\choose 2}(n-1)=O(n^3)$.

If we plug in Theorem~\ref{main} into the machinery of~\cite{PT} and \cite{KPTU}, a routine calculation gives the following.

\begin{corollary}\label{cor}\label{crossinglem}
Every separated single-crossing topological multigraph on $n$ vertices and $e \geq 4n$ edges has at least
$10^{-25}\frac{e^3}{n^2\log(e/n)}$ crossings.
\end{corollary}

For simplicity, we will assume that a multigraph does not have \emph{loops}.  It is easy to see that Theorem 1 also holds for topological multigraphs  with loops, assuming that each loop contains a vertex.

Theorem~\ref{main} does not remain true if we replace the assumption that $G$ is {\em single-crossing} by the weaker one that any two edges cross at most {\em twice}. To see this, let the vertices of $G$ lie on the $x$-axis: set $V(G)=\{1,2,\ldots,n\}$. Let each edge consist of a semicircle in the upper half-plane and a semicircle below it that meet at a point of the $x$-axis. More precisely, for any pair of integers $i,j\in V(G)$ with $i<j$, and for any $k$ with $i\le k<j,$ pick a distinct point $p_{ikj}$ in the open interval $(k,k+1)$. Let $\gamma_{ikj}$ be the union of two semicircles centered at the $x$-axis: an upper semicircle connecting $i$ to $p_{ikj}$ and a lower one connecting $p_{ikj}$ to $j$. Let $E(G)$ consist of all arcs $\gamma_{ikj}$ over all triples $i\le k<j.$ Observe that any two edges of $G$ cross at most twice: once above the $x$-axis and once below it. No two parallel edges, $\gamma_{ihj}$ and $\gamma_{ikj}$ with $h<k$, cross each other, and the region enclosed by them contains the vertex $k\in V(G)$. Therefore, $G$ is a separated topological multigraph with $\sum_{i,j (i<j)}(j-i)=\Omega(n^3)$.
\smallskip

The proof of Theorem~\ref{main} is presented in the next section, and the proof of Corollary \ref{cor} can be found in the subsequent section.

All logarithms used in the sequel are of base 2. We omit all floor and ceiling signs wherever they are not crucially important.

\section{Proof of Theorem \ref{main}}

We will need the following simple lemma.

\begin{lemma}\label{thrackle}
Let $G$ be a single-crossing topological graph on $n$ vertices with no parallel edges, in which every pair of independent edges cross. Then we have $|E(G)|\leq 4n$.
\end{lemma}

\begin{proof}
Let $V(G)=A\cup B$ be a bipartition of the vertex set such that at least half of the edges of $G$ run between $A$ and $B$. Denote the corresponding bipartite graph by $G(A,B)$. Any pair of independent edges of $G(A,B)$ cross once, that is, an {\em odd} number of times. Assume without loss of generality that $A$ and $B$ are separated by a horizontal line. By ``flipping'' one of the half-planes bounded by this line from left to right, we obtain a drawing of $G(A,B)$, in which any pair of independent edges cross an {\em even} number of times. According to the Hanani-Tutte theorem~\cite{Tu70,Sch13}, this implies that $G(A,B)$ is a planar graph. Any bipartite planar graph on $n\ge 3$ vertices has at most $2n-4$ edges. Therefore, $|E(G)|\le 2|E(G(A,B))|\le 4n-8.$
\end{proof}

\begin{proof}[Proof of Theorem \ref{main}]  Let $G  = (V,E)$ be a separated single-crossing topological multigraph on $n$ vertices. If two vertices, $u$ and $v$, are joined by $j>1$ parallel edges, then they cut the plane into $j$ pieces, one of which is unbounded. The bounded pieces are called {\em lenses}. Each lens is bounded by two adjacent edges joining a pair of vertices. Let $L$ denote the set of lenses determined by $G$.

\smallskip

If $|L| \leq {|E(G)|\over 2}$, then keeping only one edge between every pair of adjacent vertices, we obtain a simple graph $G'$ whose number of edges satisfies
$${|E(G)|\over 2}\le |E(G')|\le{n\choose 2}.$$
This implies that $|E(G)|<n^2,$ and we are done.

From now on, we can and will assume that
\begin{equation}\label{one}
|L| \geq {|E(G)|\over 2}.
\end{equation}

For any lens $\ell \in L$, let $|\ell|$ denote the number of vertices in the interior of $\ell$.  For $t = \log n$, we partition $L$ into $t$ parts, $L_1\cup L_2\cup \cdots \cup L_{t}$, where $\ell \in L_i$ if and only of $2^{i-1} \leq |\ell| < 2^{i}$.  By the pigeonhole principle, there is an integer $k, 1\le k\le t,$ such that

\begin{equation}\label{two}|L_k| \geq {|L|\over\log n}.
\end{equation}

Fix an integer $k$ with the above property, and let $d_k(v)$ denote the number of lenses in $L_k$ that contain vertex $v$ in its interior.  Then we have

$$\sum\limits_{v \in V}d_{k}(v) = \sum\limits_{\ell \in L_k}|\ell| \geq |L_k|2^{k-1}.$$

Hence, there is a vertex $v \in V$ that lies in the interior of at least $|L_k|{2^{k-1}\over n}$ lenses from $L_k$.  Assume without loss of generality that $v$ is located at the origin $o$, and let $L^o$ denote the set of lenses in $L_k$ which contain the origin.  Hence, we have
$$|L^o| \geq |L_k|\cdot {2^{k-1}\over n}.$$
Combining this with (\ref{one}) and (\ref{two}), we obtain
\begin{equation}\label{three}
|L^o| \geq \frac{|E(G)|}{n\log n}\cdot 2^{k-2}.
\end{equation}

Let $G^o$ denote the subgraph of $G$ consisting of all vertices and the edges that bound a lens in $L^o$. Any two vertices of $G^o$ are connected by 0 or 2 edges of $G^o$.

\smallskip

Now we use the idea of the probabilistic proof of the crossing lemma; see \cite{Ma02}. Let $W$ be a random subset of $V$ in which each vertex is picked independently with probability $p=2^{-k}$.  Let $G^o[W]$ be the subgraph of $G^o$ induced by $W$. Let $L^o(W)$ denote the set of {\em empty} lenses in $G^o[W]$ (that is, the set of lenses with empty interiors).  For the expected values of $|W|$ and $|L^o(W)|$, we have

$$\mathbb{E}[|W|]  = pn$$
and
$$\mathbb{E}[|L^o(W)|] \ge p^2(1-p)^{2^k}|L^o|.$$

By linearity of expectation, there is a subset $W'$ of $V$ such that

\begin{equation}\label{four}|L^o(W')|-4|W'| \geq \mathbb{E}[|L^o(W)|] -4\mathbb{E}[|W|] \geq p^2(1-p)^{2^k}|L^o|-4pn.\end{equation}


For each lens in $\ell \in L^o(W')$, we arbitrarily pick one of the two edges enclosing $\ell$, and denote the resulting simple topological graph by $G'$.  We now make the following observation.

\begin{lemma}\label{geom}
Any two independent edges of $G'$ cross each other.
\end{lemma}

\begin{proof}
Suppose, for contradiction, that $G'$ has two independent edges, $e$ and $e'$, which do not cross.  Let $\ell$ and $\ell'$ be the corresponding empty lenses in $G^o[W']$. Since the interiors of $\ell$ and $\ell'$ are empty, neither of them can contain an endpoint of the other. Both of these lenses contain the origin $o$, which implies that they cannot be disjoint. Therefore, both sides of $\ell$ must cross both sides of $\ell'$, contradicting the choice of $e$ and $e'$. Here, we used the assumption that $G$ and, hence, $G'$ are single-crossing. \end{proof}

In view of Lemma~\ref{geom}, we can apply Lemma \ref{thrackle} to $G'$. We obtain $|E(G')|=|L^o(W')|\le 4|W'|$ and hence by (\ref{four}) we have $p^2(1-p)^{2^k}|L^o| \leq 4pn$.
It follows that
$$|L^o| \leq 4p^{-1}(1 - p)^{-2^k}n.$$

\noindent Substituting $p = 2^{-k}$, we get

$$|L^o| \leq 16\cdot 2^kn.$$

\noindent Comparing this with (\ref{three}), we conclude that

$$|E(G)| \leq 64n^2\log n.$$
This completes the proof of Theorem~\ref{main}.
\end{proof}

\section{Proof of Corollary \ref{crossinglem}}

The \emph{bisection width} of a graph is defined as the minimum number of edges whose deletion separates the graph into parts each containing at most a fraction $4/5$ of the vertices. For our purposes, the constant $4/5$ is not so important, and any positive constant less than $1$ will do. A bound on the bisection width of a graph in terms of the crossing number and the degree sequence of the graph was established by Pach, Shahrokhi, and Szegedy \cite{PSS} through an application of the Lipton-Tarjan separator theorem \cite{LT}. It was applied to obtain a new proof and various generalizations of the crossing lemma \cite{PST}, \cite{BCSV}, \cite{KPTU}.

The proof of Corollary \ref{crossinglem} follows the same general strategy as the new proof of the crossing lemma in \cite{PST}, although, importantly, we need to work with a topological variant of the bisection width, introduced in \cite{PT}. Given a separated single-crossing topological graph $G=(V,E)$, let $b(G)$ denote the minimum number of edges we need to delete from $G$ so that there is a vertex partition $V=V_1 \cup V_2$ such that

(i) $|V_1|,|V_2| \leq 4|V(G)|/5$, 

(ii) there is no remaining edge between $V_1$ and $V_2$, and 

(iii) for $i=1,2$, the subgraph $G_i\subset G[V_i]$ formed by the remaining edges of $G[V_i]$ is a separated single-crossing topological graph. 

Our proof of Corollary \ref{crossinglem} is based on the following statement.

\begin{lemma}\label{bwcrossing}
Let $G$ be a separated single-crossing topological graph on $n$ vertices with degree sequence $d_1,\ldots,d_n$ and $c(G)$ crossings. Then
$$b(G) \leq 22\sqrt{c(G)+\sum_{i=1}^n d_i^2 + n}.$$

\vspace{-1.2cm} \hfill $\Box$

\end{lemma}

\vspace{0.5cm} 

Lemma \ref{bwcrossing} was proved in \cite{PT}, in the special case where no two edges of $G$ incident to the same vertex cross each other. The argument in \cite{PT} goes through to the general case without any change. In fact, it also remains valid for other ``drawing styles'' (see \cite{KPTU}, Theorem 2).

\begin{proof}[Proof of Corollary \ref{crossinglem}]
Let $G$ be a separated single-crossing topological graph on $n$ vertices with $e$ edges, degree sequence $d_1,\ldots,d_n$, and $c(G)$ crossings.
Similar to (Corollary 2.2 of \cite{PT}), we have $c(G) \geq e-3n$ through an application of Euler's polyhedral formula. This bound is sufficient for $e < 10^{12}n$. We may therefore assume $e \geq 10^{12}n$.

Let $\Delta:=\lceil 2e/n \rceil$. By locally splitting each vertex of $G$ whose degree is larger than $\Delta$ into vertices of degree $\Delta$, with possibly one additional vertex of degree less than $\Delta$, we obtain another separated single-crossing topological graph $G'$ with $e$ edges, $c(G)$ crossings, and with at most $2n$ vertices of maximum degree at most $\Delta:=\lceil 2e/n \rceil$. Denoting the number of vertices of $G'$ by $n'$, we have $n \leq n' \leq 2n$.

Let $k=10^{-10}\frac{e^2}{n^2\log(e/n)}$. Note that $k \geq \Delta$, as $e \geq 10^{12}n$.

Let $\mathcal{F}_0=\{G'\}$. We consider a process where at step $i=1,2,\ldots$ we create a family $\mathcal{F}_i$ of vertex-disjoint subgraphs of $G'$, each of which is a separated single-crossing topological graph. Throughout the process, we maintain the property, true for $i=0$, that each member $H \in \mathcal{F}_i$ either satisfies $c(H) \geq ke(H)$ or has at most $(4/5)^i n'$ vertices. Assume that this property is satisfied at step $i$. At step $i+1$, we construct $\mathcal{F}_{i+1}$ from $\mathcal{F}_i$, as follows. For each $H \in \mathcal{F}_i$, if $c(H) \geq ke(H)$ or $H$ has at most $(4/5)^{i+1} n'$ vertices, then place $H$ in $\mathcal{F}_{i+1}$. Otherwise, we have $c(H) < ke(H)$ and the number of vertices of $H$ is at least $(4/5)^{i+1} n'$ and at most $(4/5)^{i} n'$. In this case, we apply Lemma \ref{bwcrossing} to $H$: by deleting $b(H)$ edges from $H$, we obtain two vertex-disjoint separated single-crossing topological subgraphs, $H_1$ and $H_2$, each of which has at most $\frac{4}{5}v(H)$ vertices. We place $H_1$ and $H_2$ in $\mathcal{F}_{i+1}$. Obviously, the resulting family $\mathcal{F}_{i+1}$ has the desired property.

Next, we bound the number of edges that were deleted from the members of $\mathcal{F}_i$ to obtain $\mathcal{F}_{i+1}$. For each $H \in \mathcal{F}_i$ which does not belong to $\mathcal{F}_{i+1}$, we have $c(H) < ke(H)$. It follows from Lemma \ref{bwcrossing} that the number of edges we deleted from $H$ to obtain $H_1$ and $H_2$ is at most
$$b(H) \leq 22\sqrt{c(H)+\Delta e(H)+v(H)} \leq 22\sqrt{ke(H)+\Delta e(H)+v(H)} \leq 40\left(\sqrt{ke(H)}+\sqrt{v(H)}\right).$$
Since each $H \in \mathcal{F}_i$ which does not belong to $\mathcal{F}_{i+1}$ has at least $(4/5)^{i+1} n'$ vertices, the number of $H$ with this property is at most $n'/\left((4/5)^i n'\right)=(5/4)^{i+1}$. As the sum of $e(H)$ with $H \in \mathcal{F}_i$ is at most $e$, we obtain $\sum_{H \in \mathcal{F}_i \setminus \mathcal{F}_{i+1}} \sqrt{e(H)} \leq \sqrt{(5/4)^{i+1}}\sqrt{e}$. Similarly,
$\sum_{H \in \mathcal{F}_i \setminus \mathcal{F}_{i+1}} \sqrt{v(H)} \leq \sqrt{(5/4)^{i+1}}\sqrt{n'}$. Hence, in going from $\mathcal{F}_i$ to $\mathcal{F}_{i+1}$, at most
$40(5/4)^{(i+1)/2}\left(\sqrt{ke}+\sqrt{n}\right)$ edges were deleted, in total.

We stop at the last step $i_0$ such that $(5/4)^{i_0/2} \leq 10^{-3}(e/k)^{1/2}$. In total, in getting from step $0$ to step $i_0$, the number of edges we have deleted is at most
$$\sum_{i=0}^{i_0-1} 40(5/4)^{(i+1)/2}\left(\sqrt{ke}+\sqrt{n}\right) \leq 500(5/4)^{i_0/2}\sqrt{ke} \leq e/2.$$
Hence, the members of $\mathcal{F}_{i_0}$ contain, in total, at least $e/2$ edges. On the other hand, each graph $H \in \mathcal{F}_{i_0}$ either satisfies $c(H) \geq ke(H)$ or has at most $v:=(4/5)^{i_0}n' \leq 3 \cdot 10^{6}kn/e$ vertices. By Theorem~\ref{main}, each graph $H$ of the latter type has at most $64v^2\log v$ edges. In total, the number of edges in all members $H$ in $\mathcal{F}_{i_0}$ of the latter type is at most $64n'v\log v \leq 2 \cdot 10^9 kn^2e^{-1}\log(n/e) = e/5$. Hence, there are at least $e/2-e/5 = 3e/10$ edges of $G'$ which are in graphs $H$ in $\mathcal{F}_{i_0}$ with $c(H) \geq ke(H)$. Therefore, in these subgraphs altogether we have at least $k(3e/10) \geq 10^{-11}\frac{e^3}{n^2\log(e/n)}$ crossings. All of them are distinct crossings in the original graph $G$, which completes the proof. \end{proof}

\end{document}